\documentclass[oneside]{amsart}
\usepackage[latin1]{inputenc}
\usepackage[T1]{fontenc}
\usepackage{amstext}
\usepackage{amsmath, amssymb}
\usepackage[english]{babel}
\usepackage{amsfonts}
\newcommand{\comment}[1]{}

\newcommand{\R}{\mathbb{R}}

\newcommand{\N}{\mathbb{N}}

\newcommand{\E}{\mathbb{E}}

\newcommand{\Z}{\mathbb{Z}}

\newcommand{\pp}{\mathbb{P}}
\newcommand\kA{\mathcal{A}}
\newcommand\kB{\mathcal{B}}
\newcommand\kE{\mathcal{E}}

\newcommand\kK{\mathcal{K}}

\newcommand\kF{\mathcal{F}}
\newcommand\kG{\mathcal{G}}
\newcommand\kH{\mathcal{H}}
\newcommand\kI{\mathcal{I}}

\newcommand\kL{\mathcal{L}}

\newtheorem {theo} {Theorem} [section]

\newtheorem{lem}[theo]{Lemma}
\newtheorem{prop}[theo]{Proposition}

\newtheorem{expl}[theo]{Example}

\newcommand\si{\sigma}

\def\Ga{\Gamma}
\def\Th{\Theta}
\newcommand\beq{\begin{equation}}
\newcommand\eeq{\end{equation}}
\newcommand\p{\varphi}

\def\l{\ell}
\def\ep{\varepsilon}
\def\al{\alpha}

\def\wt{\widetilde}
\def\wh{\widehat}

\numberwithin{equation}{section}
\title{Random walks with occasionally modified transition probabilities}

\author{Olivier Raimond}
\address{Laboratoire Modal'X, Universit\'e Paris Ouest Nanterre La D\'efense, B\^atiment G, 200 avenue de la R\'epublique 92000 Nanterre, France.}
\email{olivier.raimond@u-paris10.fr}
\author{Bruno Schapira}
\address{D\'epartement de Math\'ematiques, B\^at. 425, Universit\'e Paris-Sud 11, F-91405 Orsay, cedex, France. }
\email{bruno.schapira@math.u-psud.fr}


\begin{document}
\begin{abstract} We study recurrence properties and the validity of
the (weak) law of large numbers for (discrete time) processes which, in
the simplest case, are obtained from simple symmetric random walk on
$\Z$ by modifying the distribution of a step from a fresh point.
If the process is denoted as $\{S_n\}_{n \ge 0}$, then the conditional
distribution of $S_{n+1} - S_n$ given the past through time $n$ is
the distribution of a simple random walk step, provided $S_n$ is at a
point which has been visited already at least once during $[0,n-1]$.
Thus in this case $P\{S_{n+1}-S_n = \pm 1|S_\ell, \ell \le n\} =
1/2$. We denote this distribution by $P_1$. However, if $S_n$ is at a
point which has not been visited before time $n$, then we take for
the conditional distribution of $S_{n+1}-S_n$, given the past, some
other distribution $P_2$. We want to decide in specific cases whether
$S_n$ returns infinitely often to the origin and whether $(1/n)S_n
\to 0$ in probability. Generalizations or variants of the $P_i$ and
the rules
for switching between the $P_i$ are also considered.
\end{abstract}
\maketitle

\section{Introduction}
There have been a number of investigations of recurrence/transience
of "slightly perturbed" random walks. Roughly speaking we are
thinking of processes (in discrete time) whose transition
probabilities are "usually" equal to a given transition probability,
but "occasionally" make a step according to a different transition
probability.
Arguably the most challenging of these problems is the
question of recurrence vs transience of "once reinforced" simple
random walk on $\Z^d$. In the vertex version of this process, the
walk moves at the $\l$-th step from a vertex $x$ to a neighbor $x+y$
with a probability proportional to a weight $w(\l,x+y)$. All these
weights start out with the value 1, but then $w(\l,x+y)$ is increased
to $1+C$ for a given constant $C > 0$ at the first time $\l$ at which
the walk visits $x+y$. After this change the weight of $x+y$ does not
change, that is, $w(m,x+y) = 1+C$ for all $m$ greater than the time
of the first visit to $x+y$ by the process. In general little is
known so far about recurrence or transience of such processes (except
on $\Z$; see \cite{MR} for some recent results).
Other examples include excited or cookie random walks on $\Z^d$, introduced by
Benjamini and Wilson \cite{BW}, which at first visits to a site have a bias in
some fixed direction and at further visits
make a simple random walk step. These
processes have now been well studied in dimension $1$
(see \cite{D} and \cite{KM} for recent results and references therein),
but much less is known in higher dimension.

Benjamini proposed the study of random walks which are perturbed in a
somewhat different manner. We describe a slightly generalized version
of his setup. Let $P_1, P_2,\dots,P_k$, be $k \le \infty$ probability
distributions on $\R$ or on $\Z$, with zero-mean if they have finite first moment, or symmetric. Intuitively speaking we now
consider a process $S_n = S_0 + \sum_{\l=1}^n X_\l, ~n \ge 0$, for
which the $X_\l$ are chosen in two steps. First we choose an index
$i(\l) \in \{1,2,\dots,k\}$ and then, given the past through time $\l-1$
and $i(\l)$, the conditional distribution of $X_\l$ is taken to be
$P_{i(\l)}(\cdot)$. More formally, if we set
$\kH_n=\sigma((S_\l,i(\l))_{\l\le n})\vee \sigma(S_0)$, then we have for all $n\ge 0$:
\begin{equation}
\label{loiconditionnelle}
\hbox{ the conditional law of } X_{n+1} \hbox{ given } \kH_n\vee
\sigma(i(n+1)) \hbox{ is } P_{i(n+1)}.
\end{equation}
Condition \eqref{loiconditionnelle} is not enough to describe the law of
$(S_0,(S_n,i(n))_{n\ge 1})$. This law will be completely described once the way
 the sequence $i(n)$ is chosen will be given, or equivalently once the
 conditional law of $i(n+1)$ given $\kH_n$ will be given. One way is to
 choose $i(n+1)$ such that it is $\kH_n$-measurable, in which case there
 exists $f: \R\times\cup_{n\ge 0} (\R\times \{1,\dots,k\})^n \to \{1,\dots,k\}$
 such that $i(n+1)=f(S_0,(S_\l,i(\l))_{\l\le n})$. In general, there can be added
 an extra randomness in the choice of $i(n+1)$, in which case, the
 conditional law of $i(n+1)$ given $\kH_n$ is a law $\nu_n$
 which is a function of $(S_0,(S_\l,i(\l))_{\l\le n})$. Such laws can be described
 by mean of a random variable
 $A_{n+1}$ uniformly distributed on $[0,1]$, independent of $\kH_n$,
 and a function $F : [0,1] \times \R\times\cup_{n\ge 0} (\R\times \{1,\dots,k\})^n
 \to \{1,\dots,k\}$, such that $\nu_n$ is the conditional law of
 $F(A_{n+1},S_0,(S_\l,i(\l))_{\l\le n})$ given $\kH_n$. We use here the convention $(\R\times \{1,\dots,k\})^0=\emptyset$. 
 Note also that there exists a measurable function $G:[0,1]\times \{1,\dots,k\} \to \R$
 such that if $B$ is a random variable uniformly distributed on $[0,1]$,
 then $P_i$ is the law of $G(B,i)$. This function $G$ will be fixed later on.

A convenient way to construct processes satisfying \eqref{loiconditionnelle} will be
to start from independent sequences of independent random variables uniformly
distributed on $[0,1]$, $(A_n)_{n\ge 1}$ and $(B_n)_{n\ge 1}$, an independent
random variable $S_0$,  and a measurable function
$F : [0,1] \times \R\times\cup_{n\ge 0} (\R\times \{1,\dots,k\})^n \to \{1,\dots,k\}$
(which describes how we choose the law to be used for the next jump). We then define
$(S_n,i(n))$ recursively: for $n\ge 0$,
\begin{eqnarray}
\label{conditioni}
\left\{
\begin{array}{ll}
i(n+1)=F(A_{n+1},S_0,(S_\l,i(\l))_{\l\le n})\\
S_{n+1}-S_n=G(B_{n+1},i(n+1)).
\end{array}\right.
\label{unif}
\end{eqnarray}
Note that all processes $(S_0,(S_n,i(n))_{n\ge 1})$ satisfying \eqref{loiconditionnelle}
are equal in law to a process $(S_0,(S_n,i(n))_{n\ge 1})$ defined by
\eqref{conditioni} for a particular choice of function $F$. The law of this
process is thus given by the law $\mu_0$ of $S_0$, the function $F$ and
the sequence $\{P_1,...,P_k\}$. We denote the law of this process $\pp_{F,\mu_0}$ and simply by $\pp_{F,S_0}$ when $S_0$ is not random.

Another way to construct $(S_n,i(n))$  is as follows. This construction will
be used in the last section of this paper.
Fix $S_0$ in some way and let $\{Y(i,n), 1\le
i \le k,n \ge 1\}$ be a family of independent random variables such
that each $Y(i,n)$ has distribution $P_i$. These $Y(i,n)$ can be
chosen before any $i(\l)$ is determined. Now define inductively
$$
j(i,\l)= 1 + \text{ number of times $P_i$ has been used during
}[1,\l],
$$
and take for $n\ge 0$,
$$
X_{n+1}= S_{n+1} - S_n= Y\big(i(n+1),j(i(n+1),n)\big).
$$
We chose this terminology because we think of the sequence $Y(i,1),
Y(i,2),\dots$ as a supply of variables with distribution $P_i$, and
every time $i(\l)=i$ we "use" one of these variables. When we come to
pick the $Y$ variable at time $n+1$ according to $P_{i(n+1)}$ we use
the first $Y(i(n+1),\cdot)$ which has not been used yet. This is
automatically independent of all $Y$'s used by time $n$. 
We define
$$
\kF_0 = \si\text{-field generated by }S_0,
$$
and
\begin{eqnarray*}
\begin{array}{c}
\kF_{n+1} = \si\text{-field generated by }S_0,\;\kF_n,\;i(n+1),\\
\text{and }Y\big(i(n+1), j(i(n+1),n)\big).
\end{array}
\end{eqnarray*}

Once we have observed the variables which generate $\kF_n$ we first
determine $i(n+1)$ by some rule. This rule may be randomized, but
will actually be deterministic in our examples. Then
$j\big(i(n+1),n\big)$ is also determined by $i(n+1)$ and $\kF_n$-
measurable functions. Finally we determine $Y\big( i(n+1),n\big)$
and that completes the generators of $\kF_{n+1}$. The only conditions
on the rule for choosing $i(n+1)$ are that, conditionally on $\kF_n$,
all the random variables $i(n+1)$ and $\{Y(i,j),j > j(i,n), 1\le i
\le k\}$, are independent, with each $Y(i,j)$ with $j > j(i,n)$
having conditional distribution $P_i$.

Note that if all $P_i$, $i\le k$, have finite first moment (and zero-mean),
then $S_n$ automatically satisfies the strong law of large numbers, i.e.
$S_n/n\to 0$ almost surely, as soon as the tails of the $P_i$ are dominated
by some fixed distribution with finite first moment (see Lemma 1 in \cite{Kes}).
However the question of recurrence or transience of $S_n$ is much more
delicate, even when $k=2$. In particular in \cite{DKL}, Durett, Kesten and Lawler
exhibit examples where $S_n$ is transient (see also \cite{KL} for some
necessary conditions for transience).

Benjamini's questions concerned the case when $k=2$, $P_1$ puts mass
$1/2$ on each of the points $+1$ and $-1$, while $P_2$ is a symmetric
distribution on $\Z$ in the domain of normal attraction of a
symmetric Cauchy law (in particular $P_2$ does not have finite first moment).
As for the $i(n)$, Benjamini made the following choices:
$i(n+1) = 2$ if $S_n$ is at a "fresh" point, i.e., if at time $n$ the
process is at a point which it has not visited before. If $S_n$ is at
a position which it has visited before take $i(n+1) =1$. Thus his
process is a perturbation of simple random walk; it takes a special
kind of step from each fresh point but is simple random walk
otherwise. His principal questions were whether the process $\{S_n\}$
is recurrent and whether it satisfies the weak law of large numbers,
i.e., whether $(1/n)S_n \to 0$ in probability.

In Section 2.1 we present a general method to attack this kind of problems, 
which allows us to answer Benjamini's first question affirmatively (see Example \ref{ex1}). 
Our principal tool is a coupling between the $S_n$ of Benjamini's process and a
Cauchy random walk. The latter is a random walk with i.i.d. steps,
all of which have a symmetric distribution $P$ on $\Z$ which
is in the domain of normal attraction of the symmetric Cauchy law.
Unfortunately, so far our method works only for very specific $P$,
including the distribution of the first return position to the
horizontal axis of symmetric simple random walk on $\Z^2$. It
seems that even asymptotically small changes in $P$ cannot be
handled by this method.
In Section 2.2 we present an analogous method in a continuous setting,
i.e., when $P$ is a distribution on $\R$, and prove in particular that
if $P$ is the usual Cauchy law (with density $1/\pi(1+y^2)$), then
$\{S_n\}$ is recurrent (see Proposition \ref{continuousex}).

Our coupling technique permits also to give sufficient
criteria for the process $\{S_n\}$ to be transient.

In the last section we prove a weak law of large numbers for Benjamini's process.

\medskip
{\bf Acknowledgements} Part of the results presented here were obtained with Harry Kesten, and we thank him for his 
great help. We are also grateful to Itai Benjamini for
suggesting the problems considered in this paper and for his
insightful comments which helped to solve some of them.

\section{Coupling method}
In this section
one wants to construct the process $\{S_n\}$ coupled
to another process. If such a coupling exists, then $\{S_n\}$
automatically is recurrent (transience properties will also be considered). The problem now is whether the required
coupling exists. The next subsections describe the desired coupling.

\subsection{Discrete case}
\subsubsection{Successfull coupling}
The following properties, which a Markov chain with
transition matrix $Q$ on $\Z^2$ may or may not have,
will be useful. If $(U,V)$
is a Markov chain on $\Z^2$ starting from $(0,0)$, with
transition matrix $Q$, let
$$
T= T(Q):=\inf\{n>0: V(n)=0\}.
$$
If
\begin{eqnarray}
\label{Tfini} T \text{ is a.s. finite,}
\end{eqnarray}
then the law $P$ of $U_T$ is well defined. This
$P$ equals  $\p(Q)$ for some function $\p$.
The following definition may differ slightly from the definition
the reader knows. A process on $\Z$ is said to be {\it recurrent} if for any $u \in
\Z$ the process visits $u$ infinitely often. We say
that a law $P$ on $\Z$ is recurrent, if the random walk whose steps
have distribution $P$ is recurrent.

Another property is invariance under horizontal
translations, that is,
\begin{eqnarray}
\label{3.14} Q[(u,v),(u',v')]=Q[(0,v),(u'-u,v')] \quad
\text{for all }
(u,u',v,v'). \end{eqnarray}

The next property is that $Q$ can be coupled with a certain given
transition matrix $Q_0$ on $\Z^2$ in such a way that "paths
chosen according to $Q$ lie below paths chosen according to $Q_0$."
The precise meaning of this is that \eqref{3.15} below holds. Assume that
$Q$ and $Q_0$ are translation invariant in the sense of \eqref{3.14}.
We say that $Q$ can be {\it successfully coupled with} $Q_0$ if there
exists a transition matrix $\wh{Q}$ on $\Z^3$ such that
\beq \label{3.16}
\sum_{w'}\wh{Q}((u,v,w),(u',v',w'))=Q_0((u,v),(u',v')) \quad
\text{for all } (u,v,w,u',v'), \eeq
\beq \label{3.17}
\sum_{v'}\wh{Q}((u,v,w),(u',v',w')) =Q((u,w),(u',w')) \quad
\text{for all } (u,v,w,u',w'), \eeq
and
\begin{eqnarray} \label{3.15}
&\wh{Q}((u,v,w),(u',v',w'))=0 \quad \text{for all $(u,v,w,u',v',w')$ }\\
\nonumber &\qquad  \text{such that $|w|\leq |v|+1$ and $|w'|>|v'|+1$}.
\end{eqnarray}
In this case we say that $Q$ is {\it successfully coupled with $Q_0$ by $\wh{Q}$}.
Condition \eqref{3.15} implies that if $(U,V,W)$ is a Markov chain
with transition matrix $\wh{Q}$ such that $|W_0|\le |V_0|+1$, then a.s.
for all $n\ge 0$, $|W_n|\le |V_n|+1$. Condition \eqref{3.16} (resp.
\eqref{3.17}) implies that $(U,V)$ (resp. $(U,W)$) is a Markov chain
with transition matrix $Q_0$ (resp. $Q$). Note finally, even though this
will not be needed, that \eqref{3.16} with
\eqref{3.17} implies that $U$ is a Markov chain.

\vspace{0.2cm}
Let $(U,W)$ be a Markov chain on $\Z^2$ with transition matrix
$Q$. In order to prove recurrence properties, we shall need a kind of
irreducibility condition. Set
\begin{eqnarray}
\label{kB}
\kB &:= & \left\{(U,W) \text{ visits the horizonal axis at some time
$\ge 1$} \right. \\
 & & \nonumber \hspace{2.5cm} \left. \text{and does so first at the origin }\right\},
\end{eqnarray}
and for $p > 0$, write $C(p) = C(p,Q)$ for the property
\begin{eqnarray}
\label{C(p)}
\begin{array}{c}
 Q^*_\pm\{\kB\} \ge p\;,
\end{array}
\end{eqnarray}
where $Q^*_+$ (resp. $Q^*_-$) denotes the law of the Markov chain $(U,W)$
when it starts at $(U_0,W_0)=(0,1)$ (resp. when it starts at $(U_0,W_0) = (0,-1)$).
This property will be used to prove certain stopping times (the
$\tau_i$ below) are finite. We remind the reader that recurrence is
defined in the lines right after \eqref{Tfini}.
\begin{lem}
\label{irred}
Let $Q$ and $Q_0$ be two translation invariant transition matrices on
$\Z^2$, in the sense of \eqref{3.14}, such that $Q$ is successfully coupled with $Q_0$. Assume that \eqref{Tfini} holds for $Q_0$.
If $\p(Q_0)$ is recurrent, and $\eqref{C(p)}$ holds for $Q$ for some $p>0$, then \eqref{Tfini} holds for $Q$ and $\p(Q)$ is recurrent.
\end{lem}

\begin{proof} Assume that $Q$ is successfully coupled with $Q_0$ by some $\wh{Q}$. Let $(U,V,W)$
be a Markov chain on $\Z^3$ with transition matrix $\wh{Q}$ starting at $(0,0,0)$ and denote by $\pp$ the law of this Markov chain.
Let $u$ be arbitrary in $\Z$. Since $\p(Q_0)$ is recurrent, $(U_n,V_n)=(u,0)$ infinitely often $\pp$-a.s. But
since $Q$ is successfully coupled with $Q_0$, it must hold at every time $n$ at
which $(U_n, V_n) = (u,0)$, that $|W_n| \le 1$. This implies that
the event $\kE_n :=\big \{U_n = u,W_n \in \{-1,0,+1\}\big \}$ occurs
infinitely often $\pp$-a.s.
Let $\si(1)< \si(2)\dots$ be the sequence of the successive times at which $\kE_t$
occurs and define the $\si$-fields
$$
\kK_t = \si((U_n,W_n); n \le t),\; \kL_t = \kK_{\si(t)}.
$$
Further, define the events
$$
\kB_n = \{W_{\si(n)} = 0\} = \{(U_{\si(n)},W_{\si(n)})= (u,0)\},
$$
and
\beq
\begin{split}
&\wt{\kB}_n  = \kB_n \cup \kB_{n+1} \\
&= \kB_n \cup \{W_{\si(n)} = +1,
W_{\si(n+1)} = 0\} \cup \{W_{\si(n)} = -1,W_{\si(n+1)} = 0\}.
\end{split}
\eeq
We shall complete the proof by showing that
\beq
\wt{\kB}_n \text{ occurs infinitely often $\pp$-a.s.}
\label{sufficient}
\eeq
Clearly this suffices for recurrence, since
$$
\cup_{n\ge 1} \wt{\kB}_n = \cup_{n \ge 1 } \kB_n.
$$
Now $\wt {\kB}_n \in \kL_{n+1}$. Moreover it holds
\begin{eqnarray*}
\label{BCgeneral}
\pp\{\wt{\kB}_n\mid
\kL_n\} &= & 1_{\{W_{\si(n)=0}\}}+\pp\{W_{\si(n+1)}=0 \mid \kL_n\} 1_{\{W_{\si(n)=1}\}}\\
 & &  + ~ \pp\{W_{\si(n+1)}=0 \mid \kL_n\} 1_{\{W_{\si(n)=-1}\}}\\
& \ge & \nonumber 1_{\{W_{\si(n)=0}\}}+ Q_+^*\{\kB\}  1_{\{W_{\si(n)=1}\}} + Q_-^*\{\kB\}  1_{\{W_{\si(n)=-1}\}} \\
 & & \hspace{3cm} \text{(by Markov property and translation invariance)} \\
& \ge & \nonumber p \quad \text{(by \eqref{C(p)}).}
\end{eqnarray*}
Consequently,
$$
\sum_{n \ge 1} \pp\{\wt {\kB}_n \mid\kL_n\} \ge \sum_{n \ge
1}p = \infty.
$$
The conditional Borel-Cantelli lemma (Theorem 12.15 in \cite{W}) now implies that
\eqref{sufficient} holds.
\end{proof}

This lemma proves recurrence of (the trace on the horizontal axis of) a Markov chain which uses only one
transition matrix $Q$. Benjamini's process is built up by concatenating excursions
from Markov chains with more than one transition matrix.
We shall use arguments very similar to the preceding lemma, but
involving different transition matrices, in Theorem \ref{pA}.

\subsubsection{Coupling of a modified walk with a Markov process in $\Z^2$}
Throughout this subsection we let $(Q_i, 0 \le i\le k)$, with $k\le \infty$,
be a sequence of transition
matrices on $\Z^2$,  translation invariant in the sense of \eqref{3.14},
and such that for all $i\in [1,k]$, $Q_i$ is successfully coupled with $Q_0$ 
by some $\wh{Q}_i$. We assume that 
for some $p>0$ independent of $i$,  \eqref{C(p)} holds
for all $Q_i$, $i\le k$. 
Note that by Lemma \ref{irred}, if $Q_0$ satisfies \eqref{Tfini} and if $\p(Q_0)$ is recurrent, then the $Q_i$ for $1\le i \le k$ automatically satisfy \eqref{Tfini} as well. Set $P_i=\p(Q_i)$ for all $i\le k$, and let $F$ and $G$ be the functions as defined
in the introduction. We denote by $(S_n,i(n))$ the process defined by \eqref{conditioni}.

Let us now define the coupling between the (generalized version of) the
Benjamini process $\{S_n\}$ and the Markov process with transition matrix $Q_0$
on $\Z^2$. In order to carry this out, we
note that for all $i$, there exists $G_i:\Z^3\times
[0,1]\to \Z^3$, such that if $R$ is a uniformly
distributed random variable on $[0,1]$, then
$$
\wh{Q}_i((u,v,w),(u',v',w'))=\pp[G_i((u,v,w),R)=(u',v',w')].
$$
Here and in the sequel we write $\pp$ for the measure governing the
choice of one or several uniform random variables on $[0,1]$. It
will be clear from the context to which random variables this
applies. Let $\N_k:=\{1,...,k\}$ and define
$\wh{F}:\Z \times \N_k \times \Z \times \cup_{n\ge 0}(\Z\times \N_k)^n \times [0,1]\to
\N_k$, by
\beq
\label{3.18} \wh{F}((w,i),u_0,(u_\l,i_\l)_{\l \le n},a)=i\quad \text{ if } |w|\geq 1.
\eeq
and by
\beq \label{3.18b}
\wh{F}((0,i),u_0,(u_\l,i_\l)_{\l\le n},a) = F(a,u_0,(u_\l,i_\l)_{\l\le n}). \eeq

This function $\wh{F}$ determines
the index $i$ in $Q_i$ which will govern the steps in our modified
random walk over a certain random time interval, as we make more
precise now. Let $(A_\l)_{\l\geq 1}$ and $(B_\l)_{\l\geq 1}$ be two
independent sequences of i.i.d., uniformly distributed random
variables on $[0,1]$. Let $U_0$ be a random variable distributed like $S_0$, independent of
$(A_\l)_{\l\geq 1}$ and $(B_\l)_{\l\geq 1}$. 
Let
$$
\mathcal{F}_n=\sigma((A_\l,B_\l);\l\leq n)\vee
\sigma(U_0).
$$
Define $\wh{U}_\l=(U_\l,V_\l,W_\l)$ and
$I_\l$ for $\l\geq 1$ by the following: set $\tau_0=0$, $V_0=W_0=0$ and
for $n \ge 1$
\beq
\tau_{n+1}=\inf\{\l >\tau_n:W_\l=0\}.
\label{taun}
\eeq
In Lemma \ref{lem3.2} we shall show that if $\p(Q_0)$ is recurrent, and the $Q_i$ satisfy
\eqref{C(p)}, then these stopping times are $\pp$-a.s. finite. For $m\geq 0$, set (one can take $I_0=1$, or any other value, since $W_0=0$, $I_1$ will not depend on $I_0$)
\beq
\label{eq:IM}
I_{m+1}=\wh{F}\left((W_m,I_m),U_0,(U_{\tau_\l},I_{\tau_\l})_{\{\tau_{\l}\le m\}},
A_{m+1}\right),
\eeq
with $\wh{F}$ as defined in \eqref{3.18}, and
\beq
\label{eq:UM}
\wh{U}_{m+1}=G_{I_{m+1}}(\wh{U}_m,B_{m+1}).
\eeq
Note that
$(\wh{U}_m,(\tau_l)_{\{\tau_l\le m\}},I_m)$ is
$\mathcal{F}_m$-measurable. Note also that \eqref{3.18} implies that $I_{m+1}=I_m$ when $|W_m|\ge 1$. This ensures that
for all $m\in[\tau_\l+1,\tau_{\l+1}]$, $I_m=I_{\tau_\l+1}$.
\begin{lem}
\label{lem31}
The process $(U_n,V_n)_{n\geq 0}$ is a Markov
chain with transition matrix $Q_0$.
\end{lem}
\begin{proof} For $n\geq 0$ and $(u',v')$ in $\Z^2$,
\begin{eqnarray*}
\pp\{(U_{n+1},V_{n+1})=(u',v')\mid \mathcal{F}_n\} &=&
\sum_{w',i}\pp\{\wh{U}_{n+1}=(u',v',w') \hbox{ and } I_{n+1}=i\mid
\mathcal{F}_n\}\\
&=& \sum_{w',i}\pp\{G_i(\wh U_n,B_{n+1})=(u',v',w')
\qquad \text{and } \\ &&\quad \wh{F}((W_n,I_n),U_0,(U_{\tau_\l}, I_{\tau_\l})_{\{\tau_\l \le n\}},A_{n+1})=i \mid
\mathcal F_n\}\\
&=& \sum_{w',i} \wh {Q}_i(\wh {U}_{n},(u',v',w'))
\pp\{I_{n+1}=i\mid \mathcal{F}_n\}\\
&=& Q_0((U_{n},V_{n}),(u',v')),
\end{eqnarray*}
from which we deduce the Markov property.
\end{proof}

\begin{lem} \label{lem3.2}
Let $(U_m,I_m)$ be the process defined by \eqref{eq:IM} and \eqref{eq:UM}.
Assume that all $Q_i$, $1 \le i \le k$, satisfy \eqref{C(p)} and are
successfully coupled with $Q_0$. Finally, assume that \eqref{Tfini} holds for $Q_0$ and that $\p(Q_0)$ is
recurrent. Then $\pp$-a.s. it holds $\tau_n < \infty$ for all $n \ge 0$. In particular all $Q_i$, $1\le i\le k$, satisfy \eqref{Tfini}, and we set $P_i=\p(Q_i)$.
Moreover
\newline
{\bf (a)} For all $n\geq 0$, the law of $U_{\tau_{n+1}}-U_{\tau_n}$ given
$\kG_n :=\kF_{\tau_n}\vee \sigma(I_{\tau_n+1})$ is
$P_{I_{\tau_n+1}}$.
\newline {\bf (b)} For all $n\ge 0$,
\beq
I_{\tau_{n+1}}=I_{\tau_n+1}=F(A_{\tau_n+1},U_0,(U_{\tau_\l},
I_{\tau_{\l}})_{\{\tau_\l\le n\}}).
\label{index}
\eeq
\newline
{\bf (c)} For $(i,u)\in\{1,...,k\}\times\Z$, write $\pp_{i,u}$ for the  law of
the Markov chain on $\Z^2$ with transition
matrix $Q_i$, starting from $(u,0)$ and
stopped at the first time the $w$-coordinate returns to $0$.
For all $n\ge 0$, given $\kG_n$, the law of the excursion from the $U$-axis
$$(U_{\tau_n+\l},W_{\tau_{n+\l}})_{0\leq \l\leq \tau_{n+1}-\tau_n}$$
is $\pp_{I_{\tau_n},U_{\tau_n}}$.
\end{lem}

\begin{proof} The proof is by induction on $n$. First take $n=0$.
Then $\tau_1 < \infty$ a.s. by virtue of Lemma \ref{irred}.
Now part (a) for $n = 0$ is contained in part (c) for $n=0$.
Part (b) for $n=0$ follows from \eqref{3.18}, \eqref{3.18b} and \eqref{eq:IM}.
In particular, it follows from \eqref{3.18} and from the definition of the $\tau$'s
that $I_m$ can only change when $W_m = 0$, so that $I_m$
is constant on the intervals $[\tau_n+1, \tau_{n+1}]$ for $n=0$. Equation \eqref{index}
follows from \eqref{3.18} and \eqref{eq:IM}. The proof of part (c) for $n=0$
is very similar to the
one of Lemma \ref{lem31}. We skip the details.

Now assume that $\tau_N < \infty$ and parts (a)-(c) have been
proven for $n \le N$. Then given $\kG_n$, on the event $\{I_{\tau_N+1}=i\}$, $\tau_{N+1}-\tau_N$ is equal in law to $\tau_1$ for the Markov chain with transition matrix $Q_{i}$ started at $(U_{\tau_N},0)$. Lemma \ref{irred} implies that this $\tau_1$ is finite a.s.
Thus $\tau_{N+1} < \infty$ $\pp$-a.s. Now statements (a)-(c) for $n = N+1$ can be
proven as in the case $n=0$. Again we skip the details.
\end{proof}

The following lemma is almost immediate from Lemma \ref{lem3.2} and
the strong Markov property.
The lemma shows that a sample
path of Benjamini's process can be built up from a sequence of
excursions, by identifying the initial point of each excursion
with the endpoint of the preceding excursion.
This leads to our principal recurrence result, Theorem \ref{pA},
which deduces recurrence of a Benjamini process from simple and known recurrence
properties of some of the excursions.

\begin{lem} \label{lem=enloi}
The processes $(S_0,(S_n,i(n))_{n\ge 1})$ defined by \eqref{conditioni} and
$(U_0,(U_{\tau_n},I_{\tau_n})_{n\ge 1})$ have the same distribution.
\end{lem}

\subsubsection{Recurrence properties and examples}

\begin{theo}
\label{pA}
Let $(Q_i, 0 \le i\le k)$ be a sequence of transition matrices on $\Z^2$
which are translation invariant in the sense of \eqref{3.14}. Assume that
for all $1 \le i \le k, Q_i$ is successfully coupled with $Q_0$. Assume
further that $Q_0$ satisfies \eqref{Tfini}, $P_0=\p(Q_0)$ is recurrent and that all $Q_i$,
$1\le i \le k$, satisfy \eqref{C(p)} for
some $p>0$, independent of $i$. Then for any process $(S_0,(S_n,i(n))_{n\ge 1})$ that satisfies
\eqref{loiconditionnelle} with $P_i = \p(Q_i)$, $\{S_n\}_{n \ge 0}$ is recurrent.
\end{theo}
\begin{proof} Let $(S_0,(S_n,i(n))_{n\ge 1})$ be a process
satisfying \eqref{loiconditionnelle}.
Without loss of generality we can assume that $S_0=0$. Such process can be defined by \eqref{conditioni} for some functions $F$ and $G$.
Let $(\wh{U}_m,I_m)$ be the process defined by \eqref{eq:IM} and \eqref{eq:UM} with
$\wh{F}$ defined by \eqref{3.18} and \eqref{3.18b}, and with $\wh{U}_0=(0,0,0)$.
Let $\pp$ be the measure governing the choice of the independent uniformly
distributed random variables used to define the process $(\wh{U}_m,I_m)$.
We still denote by $\tau_n$, $n\ge 0$, the successive return times to $0$ of
$W$, as defined in \eqref{taun}.
Lemma \ref{lem=enloi} states that $(S_n,i(n))_{n\ge 0}$ is equal in law to $(U_{\tau_n},I_{\tau_n})_{n\ge 0}$.

We now prove that $\{U_{\tau_n}\}$ is recurrent on $\Z$. To this end
observe that $(U,V)$ is a Markov chain with transition matrix given
by $Q_0$ and that $P_0 = \p(Q_0)$ is  recurrent. This implies that
for any fixed $u$, $(U_\l,V_\l)=(u,0)$ for infinitely many $\l$
with $\pp$-probability 1. Moreover, by construction, $|W|\le |V|+1$. So
\begin{eqnarray*}
(U_\l,W_\l)\in \{(u,0),(u,-1),(u,1)\} \text{ infinitely often},
\end{eqnarray*}
still with $\pp$-probability 1.
Denote by $\sigma_n$, $n\ge 0$, the successive return times to $\{(u,0),(u,-1),(u,1)\}$
of $(U,W)$.

>From here on we can follow the proof of Lemma \ref{irred} (which is the case $k=1$). We redefine
$$
\kK_t := \si((U_n,W_n,I_n); n \le t),\; \kL_t = \kK_{\si(t)},
$$
and we replace the condition \eqref{C(p)} by (with the event $\kB$ as in \eqref{kB})
\begin{eqnarray}
\label{general}
\begin{array}{c}
 Q_{i,\pm 1}\{\kB\} \ge p\;,
\end{array}
\end{eqnarray}
where $Q_{i,1}$ (resp. $Q_{i,-1}$) denotes the law of the Markov chain with transition matrix $Q_i$
when it starts at $(0,1)$ (resp.
at $(0,-1)$). We further  redefine the events
$$
\kB_n = \{W_{\si(n)} = 0\} = \{(U_{\si(n)},W_{\si(n)})= (u,0)\},
$$
and $\wt{\kB}_n  = \kB_n \cup \kB_{n+1}$. The proof will be complete if we show that
$$
\wt{\kB}_n \text{ occurs infinitely often $\pp$-a.s.}
$$
Now $\wt {\kB}_n \in \kL_{n+1}$ and on $\{I_{\si(n)}=i\}$, it holds
\begin{eqnarray*}
\pp\{\wt{\kB}_n\mid
\kL_n\} &= & 1_{\{W_{\si(n)=0}\}}+\pp\{W_{\si(n+1)}=0 \mid \kL_n\} 1_{\{W_{\si(n)=1}\}}\\
 & &  + \pp\{W_{\si(n+1)}=0 \mid \kL_n\} 1_{\{W_{\si(n)=-1}\}}\\
& \ge & \nonumber 1_{\{W_{\si(n)=0}\}}+ Q_{i,1}\{\kB\}  1_{\{W_{\si(n)=1}\}} + Q_{i,-1}\{\kB\}  1_{\{W_{\si(n)=-1}\}},
\end{eqnarray*}
by using that given $\kL_n$ and on $\{I_{\si(n)}=i\}$, the law of $(U_{ \si(n)+k},W_{ \si(n)+k})_k$ stopped at the first positive time $W$ reaches $0$, is the same as the law of the Markov chain with transition matrix $Q_i$ starting at $(u,W_{\si(n)})$ and stopped at the first time $W$ reaches $0$, and then by using the  translation invariance of $Q_i$.
Next \eqref{general} implies
$$
\pp\{\wt{\kB}_n\mid \kL_n\}\ge p.
$$
We conclude by using the conditional Borel-Cantelli lemma as in the proof of Lemma \ref{irred}.
\end{proof}

We state now an analogous result which can give examples of transient processes.
We say that a process is {\it transient} if almost surely it comes back a finite
number of times to each site. A law is said to be transient if the associated
random walk is transient.

\begin{theo}
\label{theotransient}
Let $(Q_i, 0 \le i\le k)$ be a sequence of transition matrices on $\Z^2$ which
are translation invariant in the sense of \eqref{3.14} and satisfy \eqref{Tfini}. Assume that for
all $1\le i\le k$, $Q_0$ is successfully coupled with $Q_i$.
Assume further that $P_0=\p(Q_0)$ is transient.
Then for any process $(S_0,(S_n,i(n))_{n\ge 1})$, which satisfies
\eqref{loiconditionnelle} with $P_i =\p(Q_i)$, $\{S_n\}_{n \ge 0}$ is transient.
\end{theo}
\noindent The proof of this result is analogous to the proof of Theorem \ref{pA} and
left to the reader. Note the asymmetry. The hypothesis is that
$Q_0$ is successfully coupled with $Q_i$, instead of $Q_i$ with
$Q_0$.

\vspace{0.2cm}

\noindent Theorem \ref{pA} solves in particular the recurrence part in Benjamini's
original question. This is explained in the following example.
Here and in the remainder of this paper "Cauchy law" will always
be short for "symmetric Cauchy law".

\begin{expl}\label{ex1}
\em
Let $Q_0$ be the transition matrix of a simple random
walk on $\Z^2$: $Q_0((u,v),(u',v'))=1/4$ if
$(u',v')\in\{(u,v\pm 1),(u\pm 1,v)\}$.  We call $P_0:=\p(Q_0)$ the "discrete Cauchy law".
Observe that $P_0$ is recurrent.
Benjamini's process uses in an arbitrary order jumps of law $P_0$ and jumps of
law $P_1$, with $P_1(1)=P_1(-1)=1/2$.
Proving Benjamini's process is recurrent using Theorem \ref{pA} would require
finding $Q_1$ such that $P_1=\p(Q_1)$, and then to prove that $Q_0$ and $Q_1$
are both successfully coupled with $Q_0$. Such $Q_1$ does not exist.
So instead we will define $\wt{Q}_1$, such that $\wt{P}_1=\p(\wt{Q}_1)$ satisfies
\beq
\label{wtP1}
\wt{P}_1\{\pm 1\}=1/4 \text{ and }\wt{P}_1\{0\}=1/2\, .
\eeq
As far as recurrence is concerned, there is no difference between using $P_1$ or $\wt P_1$,
as we show in Lemma \ref{lemP0} below.
\end{expl}

So let us now define $\wt{Q}_1$ and the two different couplings.
Assume $(u,v,w)\in\Z^3$ are given. Let $(U,V)$ be a simple random walk on $\Z^2$
starting from $(u,v)$ and define the process $W$ by $W_0=w$ and $W_n=0$ for all
$n>0$. $W$ is deterministic and hence independent of $(U,V)$.
Then $\wh{U}=(U,V,W)$ and $(U,W)$ are Markov chains and
$\wt Q_1$, the transition matrix of $(U,W)$, has entries
$$
\wt Q_1[(u,w),(u\pm 1,0)] = 1/4 \text{ and }
\wt Q_1[(u,w),(u,0)] =1/2 \text{ for all } u,w.
$$
Moreover, it is straightforward that $\wt{Q}_1$ is successfully coupled with $Q_0$
and satisfies \eqref{C(p)} for $p=1/2$\,. Observe also that \eqref{wtP1} holds, as claimed.

Next we define the coupling of $Q_0$ with itself. We still let $(U,V)$ be a simple
random walk on $\Z^2$ starting from $(u,v)$. But this time $W$ is defined by
$W_0=w$ and for $n\ge 0$, by
\beq
W_{n+1}-W_n =\begin{cases} V_{n+1}-V_n & \text{ if } W_n V_n> 0 \text{ or if }
V_n=0 \text{ and }W_n>0, \\
- (V_{n+1}-V_n) & \text{ otherwise}.
\end{cases}
\label{jump1}
\eeq
Then $\wh{U}=(U,V,W)$ and $(U,W)$ are Markov chains, and the transition matrix
of $(U,W)$ is $Q_0$. Moreover, it is straightforward that this gives a successful
coupling of $Q_0$ with itself, and that $Q_0$ satisfies \eqref{C(p)} for $p=1/4$.
Thus the hypotheses of Theorem \ref{pA} are satisfied by $(\wt{P}_1,P_2)$,
where $P_2=P_0=\p(Q_0)$.  Therefore for all processes $(S_0,(S_n,i(n))_{n\ge 1})$
satisfying \eqref{loiconditionnelle} (or equivalently, for all
processes $(S_0,(S_n,i(n))_{n\ge 1})$ defined by \eqref{conditioni}),
the resulting processes $S$ will be recurrent. The fact that Benjamini's process
is recurrent is now a consequence of Lemma \ref{lemP0}.

\medskip
\noindent
{\bf Remark:} In Section 4 we shall use some consequences of this
example in the special case when $k=2$ and the corresponding
distributions $\wt Q_1$ and $Q_2 =Q_0$ are as defined a few lines before \eqref{jump1}.
Let now $I_n$ and $(U_n,V_n,W_n)$ be the processes defined by \eqref{eq:IM} and \eqref{eq:UM}.
Recall $(U_n,V_n)$ is a simple random walk on $\Z^2$.
First it needs to be pointed out that in this special case, the
function $G_2$ can be defined such that \eqref{jump1} is valid for
$n \in [\tau_\l, \tau_{\l+1})$ for some $\l$ with
$I_{n+1}=I_{\tau_\l+1} = 2$, and the function $G_1$ is defined such
that when $I_{n+1}=1$, then $W_{n+1}=W_n=0$.
We claim that
\beq
\begin{split}
U_n=V_n=0, & \,V_{n+1} = -1\text{ and }W_n \in \{-1,0,1\}\\
&\text{together imply }
U_n = W_n = 0 \text{ or }U_{n+1}=W_{n+1}=0.
\label{jump2}
\end{split}
\eeq
To see this assume that $U_n=V_n = 0$ and $V_{n+1} = -1$. Then $V_{n+1}-V_n = -1$.
If $W_n = 0$, then $U_n=W_n = 0$ by assumption and there is
nothing to prove. Assume then that $W_n = +1$. This excludes $I_{n+1}=1$, because when $I_{n+1}=1$, then
$W_n=W_{n+1}=0$. So $I_{n+1}=2$ and \eqref{jump1} applies. Thus
$$
W_{n+1} - W_n= V_{n+1}-V_n = -1, \text{ whence }W_{n+1} =W_n-1 =0.
$$
Moreover the jump from $(U_n,V_n)$ to $(U_{n+1},V_{n+1})$
can only be of size 1 (because $(U_n,V_n)$ is a simple random walk on $\Z^2$). But there already is a change of size 1 in the $V$-direction.
Thus we can only have $U_{n+1}-U_n=0$. This proves our claim
in case $W_n=1$. The case $W_n = -1$ is
entirely similar, since now $W_{n+1}- W_n = -(V_{n+1}-V_n).$
Thus \eqref{jump2} holds in general.

\begin{lem}\label{lemP0}
Let $(\wt{P}_i,1\le i\le k)$ be a sequence of probability distributions on $\Z$.
Assume that for all processes $(S_0,(S_n,i(n))_{n\ge 1})$ satisfying
\eqref{loiconditionnelle} with $\wt{P}_i$ instead of
$P_i$, the process $S$ is recurrent. Let now $\kI \subset \N_k$ be given and
let $(P_i,1\le i\le k)$ be defined by
$P_i=\wt{P}_i$ if $i\notin \kI$, and if $i\in \kI$,
$P_i\{u\}=\wt{P}_i\{U=u\mid U\neq 0\}$, for $u\neq 0$,
with $U$ a random variable of law $\wt{P}_i$.

Then for all processes $(S_0,(S_n,i(n))_{n\ge 1})$ satisfying \eqref{loiconditionnelle},
$S$ is recurrent as well.
\end{lem}

\begin{proof} First note that the hypothesis on the $\wt{P}_i$'s means that for any choice of
$\wt{F}:[0,1]\times\cup_{n\ge 0}(\Z\times\N_k)^n\to \N_k$, the process defined by
\eqref{conditioni} (with $\wt{F}$ and $\wt{G}$ in place of $F$ and $G$ respectively,
and $\wt{G}$ associated to the $\wt{P}_i$'s) is recurrent.

The intuition for this lemma is clear. A walker using $\wt P_i$ as
distribution for his displacement stands still with probability
$\wt P_i(0)$. In fact when he arrives at a new site he stands still
a geometric number
of times and then makes a displacement with distribution $P_i$.
The standing still has no influence on the collection of sites
visited by the walker and hence does not influence recurrence.
Recurrence will be the same whether $\wt P_i$ or $P_i$ is used.
A complication arises because we have to deal not with sequences
$(S_n)$  but with sequences $(S_n,i(n))$, and even the latter
sequences are not Markovian.

Let now $(S_0,(S_n,i(n))_{n\ge 1})$ be a process satisfying \eqref{loiconditionnelle}.
To simplify, we take $S_0=0$. As is explained in the
introduction, such a process can be constructed with functions
$F:[0,1]\times\Z\times\cup_{n\ge 0}(\Z\times\N_k)^n\to \N_k$, $G:[0,1]\times\N_k\to \Z$
and independent sequences $(A_\l)_{\l\ge 0}$ and $(B_\l)_{\l\ge 0}$ of i.i.d.
uniformly distributed random variables on $[0,1]$: for $n\ge 0$,
$$i(n+1)=F(A_{n+1},S_0,(S_\l,i(\l))_{\l\le n})$$
and
$$S_{n+1}-S_n=G(B_{n+1},i(n+1)).$$
Here $G$ is such that the law of $G(B_1,i)$ is $P_i$.

Let now $\wt{F}:[0,1]\times \Z\times \cup_{n\ge 0}(\Z\times\N_k)^n\to \N_k$ be defined by
\begin{eqnarray}
\label{wtF1}
\wt{F}(a,s(0),(s(\l),j(\l))_{\l\le n})=j(n),
\end{eqnarray}
if $j(n)\in \mathcal{I}$ and $s(n)=s(n-1)$, and otherwise by
\begin{eqnarray}
\label{wtF2}
\wt{F}(a,s(0),(s(\l),j(\l))_{\l\le n})=F(a,s(0),(s(t_\l),j(t_\l))_{\l\le m}),
\end{eqnarray}
where $t_0=0$,
\beq
\label{wtt1}
t_\l = \inf\{r\in (t_{\l-1},n]: ~ s(r)\neq s(r-1) \hbox{ or }
j(r)\notin \mathcal{I} \} \quad \text{for }\l\ge 1,
\eeq
and
\beq
\label{wtt2}
m =\sup\{\l\ :\ t_\l<\infty\}.
\eeq
Note that \eqref{wtF1}-\eqref{wtt2} are merely the {\it
definitions} of the non-random functions $t_\l, m$ and $\wt F$
at a generic point $(a,s(0),(s(\l),j(\l))_{\l \le n})$ of their domains.
Note also that, by convention, $t_\l = \infty$ if the set in the
right hand side of \eqref{wtt1} is empty. In particular this is
the case for $\l > n$.

Let $\wt{G}:[0,1]\times \N_k\to\Z$ be such that the law of $\wt{G}(B_1,i)$ is
$\wt{P}_i$. Define the random quantities $\wt{S}_0$ and $(\wt{S}_n,\wt{i}(n))$ by
$\wt{S}_0=0$ and for $n\ge 0$,
\begin{eqnarray}
\label{wti}
\wt{i}(n+1)=\wt{F}(A_{n+1},\wt{S}_0,(\wt{S}_\l,\wt{i}(\l))_{\l\le n}),
\end{eqnarray}
and
\begin{eqnarray}
\wt{S}_{n+1}-\wt{S}_n=\wt{G}(B_{n+1},\wt{i}(n+1)).
\end{eqnarray}
Equation \eqref{wtF1} implies that
\begin{eqnarray}
\label{inin+1}
\wt{i}(n+1)=\wt{i}(n)\quad  \text{if }\wt{i}(n) \in \kI \text{ and }\wt{S}_n=\wt{S}_{n-1}.
\end{eqnarray}
Let $\rho_0=0$ and $\rho_\l = \inf\{r>\rho_{\l-1}: ~ \wt{S}_r\neq \wt{S}_{r-1}
\hbox{ or } \wt{i}(r)\notin \mathcal{I}\}$ for $\l\ge 1$.
Note that $\rho_n$ is essentially the value of $t_n$ at the random place
$(\wt S_\l,\wt i(\l))_{\l \le n}$. By definition
\begin{eqnarray}
\label{Sr}
\wt{S}_r=\wt{S}_{\rho_\l} \quad \text{for all }r\in [\rho_\l,\rho_{\l+1}),
\end{eqnarray}
and
\begin{eqnarray}
\label{Sr2}
\wt{S}_{\rho_{\l+1}}\neq \wt{S}_{\rho_\l} \quad \text{if  }\wt{i}(\rho_\l+1)\in \kI.
\end{eqnarray}
Moreover, \eqref{inin+1} implies (by induction on $r$) that
$\wt{i}(r)=\wt{i}(\rho_\l+1)$ for all $r\in (\rho_\l,\rho_{\l+1}]$, and
\begin{eqnarray}
\label{wti2}
\wt{i}(\rho_{\l+1})\ =\ \wt{i}(\rho_\l+1) &=& \wt{F}(A_{\rho_\l+1},\wt{S}_0,(\wt{S}_r,\wt{i}(r))_{r\le \rho_\l})\\
&=& \nonumber F(A_{\rho_\l+1},\wt{S}_0,(\wt{S}_{\rho_r},\wt{i}(\rho_r))_{r\le \l}),
\end{eqnarray}
where the last equality follows from \eqref{wtF2}.
Now, for any $i\in \N_k$ and $u\in \Z$,
if $\kF_n=\sigma((A_\l,B_\l)_{\l\le n})$,
\begin{eqnarray}
\label{Stau}
\pp\{\wt{S}_{\rho_{\l+1}}-\wt{S}_{\rho_\l}=u\mid \wt{i}(\rho_\l+1)=i,\ \kF_{\rho_\l}\}=P_i(u).
\end{eqnarray}
Indeed when $i\in \mathcal{I}$ and $u\neq 0$, the left hand side is equal to
$$\sum_{K\ge 1} \pp\{\{\wt{G}(B_{\rho_\l+K},i)=u\}\cap\{\rho_{\l+1}-\rho_\l=K\}
\mid \wt{i}(\rho_\l+1)=i,\ \kF_{\rho_\l}\}, $$
which is equal to
$$\sum_{K\ge 1} \wt{P}_i(0)^{K-1}\wt{P}_i(u)=P_i(u).$$
If $i\in \mathcal{I}$ and $u=0$, both sides of \eqref{Stau}
equal $0$ by \eqref{Sr2} and the definition of $\kI$.
When $i(\rho_\l+1) = i$ and $i\notin \mathcal{I}$, then $\rho_{\l+1}=\rho_\l+1$
and \eqref{Stau} follows from the fact that $\wt{P}_i=P_i$.

Finally we claim that \eqref{wti2} and \eqref{Stau} show that
$(\wt{S}_{\rho_\l},\wt{i}(\rho_\l))$ has the same law as $(S_\l,i(\l))$:
indeed we shall show by induction on $\l$ that for any
sequence $j(1),\dots,j(\l) \in \N_k$, $u_1,\dots,u_\l \in \Z$,
\beq
\label{marginales}
\pp\left\{\wt{S}_{\rho_{\l+1}}-\wt{S}_{\rho_\l}=u_{\l+1},\
\wt{i}(\rho_{\l+1})=j(\l+1), \dots, \wt{S}_{\rho_1}=u_1,\wt{i}(\rho_1)=j(1)\right\},
\eeq
is equal to
\beq
\label{marginalesbis}
\pp\left\{S_{\l+1}-S_\l=u_{\l+1},\
i(\l+1)=j(\l+1), \dots, S_1=u_1,i(1)=j(1)\right\}.
\eeq
But \eqref{marginales} is equal to
\begin{eqnarray*}
 P_{j({\l+1})}(u_{\l+1}) \pp\left\{\wt{i}(\rho_{\l+1})=j(\l+1), \dots,
  \wt{S}_{\rho_1}=u_1,\wt{i}(\rho_1)=j(1)\right\},
\end{eqnarray*}
by \eqref{Stau}. By using \eqref{wti2} we see that the second
factor in this last expression is equal to
\begin{eqnarray*}
&&\pp\left\{F(A_{\rho_\l+1},s_0,(s_r,j(r))_{r\le \l})=j(\l+1)\right\} \\
&\times& \pp\left\{\wt{S}_{\rho_{\l}}-\wt{S}_{\rho_{\l-1}}=u_{\l},\
\wt{i}(\rho_{\l})=j(\l), \dots, \wt{S}_{\rho_1}=u_1,\wt{i}(\rho_1)=j(1)\right\},
\end{eqnarray*}
where for all $r$, $s_r:=u_1+\dots+u_r$. Then an induction procedure shows that
\eqref{marginales} is equal to \eqref{marginalesbis},
as claimed.

Moreover, by assumption $\wt{S}$ is recurrent and \eqref{Sr} implies that
$(\wt{S}_{\rho_\l},\l\ge 0)$ is also recurrent. This proves the lemma.
\end{proof}

\noindent We finish with this last class of examples
\begin{expl}
\label{lastexpl} {\em Take for $Q_0$ the transition matrix of a Markov chain $(U,V)$
such that $U$ and $V$ are
both Markovian and independent of each other, and such that $P_0=\p(Q_0)$ is recurrent.
Assume that for all $i\in
[1,k]$, $Q_i$ is the transition matrix of a Markov chain $(U,W_i)$ such that $U$
and $W_i$ are both Markovian and
independent of each other (the chain $U$ being the same for $Q_0$ and for $Q_i$).
Assume that all $Q_i$'s are
translation invariant (note that this hypothesis only concerns the Markov chain $U$).
Suppose also that for all
$i$ it is possible to couple the chains $V$ and $W_i$ such that $(V,W_i)$ is
Markovian and such that
if $|W_i(0)|\le|V(0)|+1$ then for all $n\ge 0$, $|W_i(n)|\le |V(n)|+1$.
 Let now  $U$ be a
chain independent of this Markov process $(V,W_i)$. Then $(U,V)$ and $(U,W_i)$ are
both Markovian respectively
with transition matrices $Q_0$ and $Q_i$. This coupling of $(U,V)$ and $(U,W_i)$
shows that $Q_i$ is successfully
coupled with $Q_0$. Assume also that the $Q_i$'s, $i\le k$, satisfy \eqref{C(p)}
for some positive $p$, uniformly in
$i$. Then the hypotheses of Theorem \ref{pA} are satisfied.

This can be applied to the following: let $(A_n)$ and $(B_n)$ be two
independent sequences of i.i.d. random variables uniformly
distributed on $[0,1]$. Let $p\in [0,1/2)$ and let $U(n)=\sum_{i=1}^n (1_{\{A_i\ge
p\}}-1_{\{A_i< 1-p\}})$. Let $V$ be the simple
random walk on $\Z$ defined by $V(0)=v_0$ and
$$V(n)-V(n-1)=1_{\{B_n< 1/2\}}-1_{\{B_n\ge 1/2\}}.$$
Let $(p_i(w): i\ge 1 \hbox{ and }w\ge 0)$
be such that $p_i(w)\in [0,1/2]$ for all $w\in \Z$.
Define $W_i$ by $W_i(0)=0$ and on the event $\{W_i(n-1)=w\}$,
\begin{eqnarray*}
W_i(n)-w
&=& \big[1_{\{B_n< p_i(w)\}}-1_{\{B_n\ge p_i(w)\}}\big]1_{\{w\ge 1\}}\\
&&  + \big[1_{\{B_n \ge p_i(w)\}}-1_{\{B_n< p_i(w)\}}\big]1_{\{w\le -1\}}\\
&&+ \big[1_{\{B_n< 1/2-p_i(0)\}}-1_{\{B_n\ge 1/2+p_i(0)\}}\big] 1_{\{w=0\}}.
\end{eqnarray*}
Then one immediately checks that $|W_i(n)|\le |V(n)|+1$ for all $n\ge 0$,
and the resulting transition matrices $Q_i$ are successfully coupled with
$Q_0$. Moreover Condition \eqref{general} is satisfied for all $Q_i$'s with
$(1-2p)/2$ instead of $p$. Thus for any such choice of $(p_i(w))$, we can
apply Theorem \ref{pA} and find in this way many examples
of recurrent processes. However, given the $(p_i(w))$'s,
it is usually not easy to describe explicitly the associated laws $P_i$.}
\end{expl}

\subsection{Continuous case}
We present now an analogous coupling method (in the spirit of Example \ref{lastexpl}) when the laws $P_i$ are defined on $\R$, because in this case, by using stochastic calculus, we can give more explicit examples of $P_i$'s, which can be used to construct recurrent processes $\{S_n\}$ (see Proposition \ref{continuousex} below).

Let $B^{(1)}$ and $B^{(2)}$ be two independent Brownian motions started at $0$.
Let $(U_0,V_0,W_0)$ be a random variable in $\R\times\R^+\times\R^+$, independent of $B=(B^{(1)},B^{(2)})$.
For all $t>0$, set $U_t=U_0+B^{(1)}_t$.
Let $(\sigma_0,b_0):[0,+\infty)\to\R^2$
be some Lipschitz functions and $v_0\ge 0$ some constant. Then (see Exercice 2.14
p.385 in \cite{RY}) the stochastic differential equation
\beq
\label{SDEV}
V_t = V_0 + \int_0^t \sigma_0(V_s)\ dB^{(2)}_s +\int_0^t b_0(V_s)\ ds +L_t,
\eeq
with $L$ the local time in $0$ of $V$, admits a unique solution which is
measurable with respect to the filtration generated by $B^{(2)}$.

Consider next $(\sigma,b):(0,+\infty)\to\R^2$ some locally Lipschitz
functions and the stochastic differential equation
\beq
\label{SDE}
W_t=W_0 + \int_0^t \sigma(W_s)\ dB^{(2)}_s +\int_0^t b(W_s)\ ds, \qquad t<T\wedge e,
\eeq
where
$$e=\inf\{t\ge 0:\ W_t=+\infty\} \quad \textrm{and}\quad  T=\inf\{t\ge 0\ :\  W_t=0\}.$$
It is known (see for instance Exercise 2.10 p.383 in \cite{RY}) that if $\sigma$
and $b$ are locally Lipschitz, then equation \eqref{SDE} admits a unique solution
$W$ which is measurable with respect to the filtration generated by $B^{(2)}$.
When $(U_0,W_0)=(0,1)$ and when $(\sigma,b)$ is such that
\begin{eqnarray} \label{Tcontfini} T<e \quad \text{ almost surely},\end{eqnarray} we
denote by $P$ the law of $U_T$. Then, like in the discrete case, we
have $P=\p(\sigma,b)$ for some function $\p$. In the following, all $(\sigma,b)$
will be assumed to be locally Lipschitz and such that \eqref{Tcontfini} is satisfied.
Moreover, for $w>0$ we will denote by $\pp^{(\sigma,b)}_{w}$ the law of $(W_t)_{t\le T}$ when $W_0=w$.

We say that $(\sigma,b):(0,\infty)\to\R^2$ is {\it successfully coupled} with
$(\sigma_0,b_0)$ if for any solutions $V$
and $W$, respectively of \eqref{SDEV} and \eqref{SDE}, with $W_0\le V_0 + 1$, we have
$W_t\le V_t+1$ for all $t<T$.
Note that, by using a comparison theorem (see \cite{IW} Theorem 1.1 p.437), if for all $v>0$,
$\sigma(v+1)=\sigma_0(v)$ and $b(v+1)\le b_0(v)$, then $(\sigma,b)$ is
successfully coupled with $(\sigma_0,b_0)$.

Let $((\sigma_i,b_i), 0 \le i\le k)$, $k\le \infty$, be a
sequence of locally Lipschitz functions on $(0,\infty)^2$ such that for all
$i\in \N_k$, $(\sigma_i,b_i)$ is successfully coupled with
$(\sigma_0,b_0)$. For $i\in \N_k$, set $P_i=\p(\sigma_i,b_i)$.

Let $F:[0,1]\times\R\times \cup_{n\ge 0}(\R\times \N_k)^{n}\to \N_k$ be given. This function $F$ determines
the index $i$ in $(\sigma_i,b_i)$ which will govern the steps in our modified
random walk over a certain random time interval, as we make more
precise now.

Let $(A_n)_{n\ge 1}$ be a sequence of independent random variables uniformly distributed
on $[0,1]$. Assume that this sequence is independent of $B$. Let $V$ be the
solution of \eqref{SDEV}, with $V_0=0$.
Let $(U_0,I_0)$ be a random variable in $\R\times\N_k$, independent of $A$ and $B$.
Define $(\tau_n)_{n\ge 0}$ an increasing sequence of random times, and the
processes $(W^*_t)_{t<\tau_\infty}$ and $(I_t)_{t<\tau_\infty}$, with
\beq
\label{tau}
\tau_\infty:=\lim_{n\to\infty}\tau_n,
\eeq
by the following:
first $\tau_0=0$. Assume then that $(\tau_1,\dots,\tau_n)$ and $(I_t,W^*_t)_{0\le t\le
\tau_n}$
are defined and measurable with respect to the $\sigma$-field
$$\sigma((U_s,B_s)_{s\le \tau_n})\vee\sigma(A_1,\dots,A_n).$$
Assume moreover that $I_t=I_{\tau_{\l+1}}$ for $t\in (\tau_\l,\tau_{\l+1}]$ and
$\l\le n-1$, and that $W^*_{\tau_\l}=0$ for all $1\le \l\le n$. For $\l\le n$, set
$i(\l)=I_{\tau_\l}$ and $S_\l=U_{\tau_\l}$.
Then we define $i(n+1)$ by
$$i(n+1)=F(A_{n+1},S_0,(S_\l,i(\l))_{\l\le n}),$$
and $W^n$ as the solution of
$$W^n_t=1+\int_0^t \sigma_{i(n+1)}(W^n_s)\ dB^{\tau_n}_s +\int_0^t b_{i(n+1)}(W^n_s)\ ds, \qquad t\le T^{(n)},$$
where $T^{(n)}$ is the first time when $W^n$ reaches $0$, and where $B^{\tau_n}_s=B^{(2)}_{\tau_n+s}-B^{(2)}_{\tau_n}$. The
process $W^n$ is well-defined since $B^{\tau_n}$ is independent of $i(n+1)$. Let
$$\tau_{n+1}:=\tau_n+T^{(n)}.$$
Then set
$$W^*_t=W^n_{t-\tau_n} \hbox{ and } I_t=i(n+1) \quad \hbox{for } t\in (\tau_n,\tau_{n+1}].$$

\noindent This defines the sequence $\tau_n$ for all $n$ and $(W^*_t,I_t)$ for $t < \tau_\infty$.

\noindent Let now $\kF_t=\sigma((U_s,B_s,I_s)_{s \le t\wedge \tau_\infty})$. Then, $(\tau_n)_{n\ge 0}$
is a sequence of $\kF_t$-stopping times and like in the discrete setting, we have
\begin{lem} \label{lem32cont} For all $n\geq 0$, the conditional law of
$U_{\tau_{n+1}}-U_{\tau_n}$ given
$\kG_n :=\kF_{\tau_n}\vee \sigma(i(n+1))$ is $P_{i(n+1)}$.
\end{lem}
\begin{proof} Given $\kG_n$, the law of $(W^n_t=W^*_{\tau_n+t})_{0\le t\le \tau_{n+1}-\tau_n}$ is $\pp^{(\sigma_{i(n+1)},b_{i(n+1)})}_1$ and $(U^n_t=U_{\tau_n+t}-U_{\tau_n})_{t\ge 0}$ is a Brownian motion independent of
$(W^n_t)_{0\le t\le \tau_{n+1}-\tau_n}$.
The lemma follows, since by definition $\p(\sigma_{i(n+1)},b_{i(n+1)})=P_{i(n+1)}$.
\end{proof}

\noindent This lemma implies that the sequence $(S_n,i(n))_{n\ge 0}$
has the same law as the process defined in the introduction by \eqref{conditioni}
(with $(S_0,i(0))=(U_0,I_0)$).
Moreover we have the following result:
\begin{prop}
\label{hyptau}
Assume that there exists positive constants $0<\alpha<1<\beta$, $\sigma_+$
and $b_+$ such that
\beq
\label{bshyp}
0\le \sigma_i(x)\le \sigma_+ \quad \textrm{and}\quad |b_i(x)|\le b_+\quad
\textrm{for all } x\in (\alpha,\beta) \quad \textrm{and all } 1\le i\le k.
\eeq
Then $\tau_\infty$, as defined in \eqref{tau}, is a.s. infinite for any
choice of $(i(n),n\ge 0)$.
\end{prop}
\begin{proof}
We start with a lemma. For $z\in\R$, let $T_z:=\inf\{t:W^*_t=z\}$. In particular, $T=T_0$.
\begin{lem}
Let $(\sigma,b)$ be locally Lipschitz and $0<\alpha<1<\beta$ some constants.
Then for all $r\ge 1$, there
exists a constant $C>0$ depending only on $r$, $\alpha$, $\beta$ and
$\sigma_{\max} :=\sup_{x\in [\alpha,\beta]} |\sigma(x)|$,
such that
$$\pp^{(\sigma,b)}_1\{T<\epsilon\} \le C \epsilon^r \qquad
\textrm{for all } \epsilon< ((1-\alpha)
\wedge (\beta-1))/(2b_{\max}),$$
where $b_{\max} :=\sup_{x\in [\alpha,\beta]} |b(x)|$.
\end{lem}
\begin{proof}
First we have
\begin{eqnarray*}
\pp^{(\sigma,b)}_1\{T<\epsilon \} \le  \pp^{(\sigma,b)}_1\{T_\alpha<\epsilon \} \le \pp^{(\sigma,b)}_1\{T_\alpha\wedge T_\beta <\epsilon \}.
\end{eqnarray*}
Next set, for all $t<T$,
$$H(t):=\int_0^t \sigma(W_s)\ dB^{(2)}_s + \int_0^t b(W_s)\ ds.$$
We have
$$\pp^{(\sigma,b)}_1\{T_\alpha\wedge T_\beta <\epsilon\}\le \pp\left\{\sup_{t\le \epsilon} \left|H(t\wedge T_\alpha\wedge
T_\beta) \right|\ge (1-\alpha)\wedge (\beta-1)\right\}.$$
If $\epsilon <((1-\alpha)\wedge (\beta-1))/(2b_{\max})$, this last term is bounded by
$$\pp\left\{\sup_{t\le \epsilon} \left|\int_0^{t\wedge T_\alpha\wedge T_\beta} \sigma(W_s)\ dB^{(2)}_s \right| \ge \frac{(1-\alpha)\wedge (\beta-1)}{2}\right\},$$
which by Doob's inequality (Theorem (1.7) p.54 in \cite{RY}) is bounded by $C\epsilon^r$ for some constant $C>0$, which depends only on $r$, $\sigma_{\max}$, $\alpha$ and $\beta$. This concludes the proof of the lemma.
\end{proof}

\noindent Taking $r=2$ in this lemma, we have that for $n > ((1-\alpha)
\wedge (\beta-1))^{-1}2b_+$
$$\pp\{\tau_{n+1}-\tau_n< n^{-1}\mid \mathcal{G}_n\}\le Cn^{-2}.$$
Proposition \ref{hyptau} follows now from
the conditional Borel-Cantelli Lemma (Theorem 12.15 in
\cite{W}) by a standard argument.
\end{proof}

Let us give now $p=(p_\epsilon)_{\epsilon>0}$ such that $p_\epsilon \in (0,1)$
for all $\epsilon>0$. Let $(\sigma,b)$ be locally Lipschitz and let $W$ be the
solution of \eqref{SDE}, with $W_0=w_0$. Remember that $T=\inf\{t>0:W_t=0\}$.
We write $C'(p) = C'(p,\sigma,b))$ for the property
\beq
\label{I'p}
\pp\{U_T\in [-\epsilon,\epsilon]\}>p_\epsilon \qquad
\textrm{for all } \epsilon>0 \quad \textrm{and all } w_0\in (0,1],
\eeq
where $U$ is a Brownian motion starting from $0$ independent of $W$.

We say that a process on $\R$ is recurrent, if for all $\epsilon >0$ and all $x \in \R$,
it returns a.s. infinitely often to $[x-\epsilon,x+\epsilon]$. Similarly a law $P$
is recurrent if the associated random walk is recurrent. The analogue of Theorem \ref{pA}
is then the following theorem:
\begin{theo}
\label{pB} Let $p=(p_\epsilon)_{\epsilon>0}$ be given. Let $(\sigma_0,b_0)$ be a Lipschitz function and $((\sigma_i,b_i),1\le i\le k)$
be a sequence of locally Lipschitz functions. Assume \eqref{Tcontfini} holds for $i\in\{0\}\cup\N_k$. Set $P_i=\p(\sigma_i,b_i)$. Assume that there exists $0<\alpha<1<\beta$ and
positive constants $\sigma_+$ and $b_+$ such that \eqref{bshyp} holds.
Assume moreover that $P_0=\p(\sigma_0,b_0)$ is a recurrent law on $\R$ and that
for each $i\in \N_k$,  $(\sigma_i,b_i)$ is successfully coupled with $(\sigma_0,b_0)$
and satisfies \eqref{I'p}. Then for any $(S_0,(S_n,i(n))_{n\ge 1})$ which satisfies
\eqref{loiconditionnelle}, the process $\{S_n\}_{n \ge 0}$ is recurrent.
\end{theo}

We state now an analogue of Theorem \ref{theotransient} which can give examples of transient processes.
We say that a process on $\R$ is transient, if for all $a<b$,
it returns a.s. a finite number of times in $[a,b]$. Similarly a law $P$
is transient if the associated random walk is transient.
\begin{theo}
Let $(\sigma_0,b_0)$ be a Lipschitz function and $((\sigma_i,b_i),1\le i\le k)$
be a sequence of locally Lipschitz functions. Assume \eqref{Tcontfini} holds for $i\in\{0\}\cup\N_k$. Set $P_i=\p(\sigma_i,b_i)$. Assume that there exists $0<\alpha<1<\beta$ and
positive constants $\sigma_+$ and $b_+$ such that \eqref{bshyp} holds.
Assume moreover that $P_0=\p(\sigma_0,b_0)$ is a transient law on $\R$ and that
for each $i\in \N_k$, $(\sigma_0,b_0)$ is successfully coupled with $(\sigma_i,b_i)$. Then for any $(S_0,(S_n,i(n))_{n\ge 1})$ which satisfies
\eqref{loiconditionnelle}, the process $\{S_n\}_{n \ge 0}$ is transient.
\end{theo}

\noindent The proof of these theorems are similar to the discrete case and left to the reader.

\medskip
\noindent As an example of laws which are successfully coupled we give the
following result:
\begin{prop}
\label{continuousex}
Let $P_0$ be the Cauchy law on $\R$. Set $(\sigma_0,b_0)=(1,0)$. Then $(\sigma_0,b_0)$ satisfies \eqref{Tcontfini}, is successfully coupled with itself, and $P_0=\p(\sigma_0,b_0)$.  Moreover for any $\alpha \in [1,2]$, there exists $(\sigma_\alpha,b_\alpha)$ locally Lipschitz satisfying \eqref{Tcontfini}, successfully coupled with $(\sigma_0,b_0)$ and such that $\p(\sigma_\alpha,b_\alpha)$ is in
the domain of normal attraction of a symmetric stable law with index $\alpha$.
\end{prop}
\begin{proof} The fact that $(\sigma_0,b_0)$ satisfies \eqref{Tcontfini} and is successfully coupled with itself is immediate (in the coupling we have $V_t=V_0+B^{(2)}_t+L_t$ and $W_t=W_0+B^{(2)}_t$ for $t<T$).
So let us concentrate on the second claim. The case $\alpha=1$ is given for
instance by $P_0$ itself. Now we prove the result for $\alpha = 2$. Take $(\sigma,b)=(0,-1)$ to be constants.
Then $W_t=W_0-t$ for all $t\le T=W_0$.  Set $P=\p(\sigma,b)=\p(0,-1)$, \eqref{Tcontfini} being obviously satisfied.
Let $U$ be a standard Brownian motion on $\R$. Observe that when $W_0=1$, then $W$ reaches $0$ at time $T=1$. Thus $P$ is the law of
$U$ at time $1$ which is the standard Gaussian and it is immediate that
$(\sigma,b)=(0,-1)$ is successfully coupled with $(\sigma_0,b_0)=(1,0)$. This gives the result for $\alpha=2$.
It remains to prove the claim for $\alpha \in (1,2)$. For $\nu \in (-1,-1/2)$,
let $W^{(\nu)}$ be a Bessel process of index $\nu$ starting from $1$, i.e.,
$W^{(\nu)}$ is the solution of the SDE:
$$
W^{(\nu)}_t=1+B^{(2)}_t+(\nu+1/2)\int_0^t\frac{1}{W^{(\nu)}_s}\ d s \text{ for all } t<T,
$$
where $B^{(2)}$ is a Brownian motion and $T$ is as always the first time when $W$
reaches $0$. It is known (see \cite{RY}) that $T$ is a.s. finite when
$\nu \in (-1,-1/2)$.
Set $\sigma^{(\nu)}=1$ and $b^{(\nu)}(w)=(\nu+1/2)/w$. Then, for $\nu \in (-1,-1/2)$,  $(\sigma^{(\nu)},b^{(\nu)})$ satisfies \eqref{Tcontfini}. Set $P^{(\nu)}=\p(\sigma^{(\nu)},b^{(\nu)})$.
We claim that if $\nu \in (-1,-1/2)$, then $(\sigma^{(\nu)},b^{(\nu)})$ can be successfully
coupled with $(\sigma_0,b_0)$ and $P^{(\nu)}$ is in the domain of attraction of a stable law
with index $-2\nu$. The first part is immediate: since $\nu+1/2\le 0$, it
follows from a comparison theorem (see \cite{IW} Theorem 1.1 p.437). For
the second part, first observe that
$$
\E\left\{e^{i u U_T}\right\}=\E\left\{e^{-\frac {u^2}2 T}\right\} \quad \text{ for all } u\in\R.
$$
So the characteristic function of $P^{(\nu)}$ is related to the Laplace
transform of $T$. For Bessel processes this last function can be expressed
in terms of modified Bessel functions: if $\phi_\nu$ is the Laplace
transform of $T$, the hitting time of $0$ for a Bessel process of index
$\nu < -1/2$ starting from $1$, then (see \cite{K} Theorem 3.1):
$$
\phi_\nu(s) = \frac{2^{\nu+1}}{\Gamma(-\nu)}\frac{K_\nu(\sqrt{2s})}{(2s)^{\nu/2}}\quad \text{ for all } s>0,
$$
where $\Gamma$ is the usual Gamma function and $K_\nu$ is a modified Bessel function (to see this from \cite{K}, take $a=1$ and let $b$ tend to $0$ in Formula $(3.7)$, and use the asymptotic when $x\to 0$ of $K_\nu(x)$ given just above Theorem 3.1). Moreover (see \cite{L} Formula (5.7.1) and (5.7.2)) we have
$$
K_\nu(s)=\frac{\pi}{2}\frac{I_{-\nu}(s)-I_\nu(s)}{\sin\nu\pi} \quad \text{ for all } s>0,
$$
where
$$
I_\nu(s)=\sum_{k=0}^\infty \frac{(s/2)^{\nu+2k}}{k!\Gamma(k+\nu+1)} \quad \text{ for all } s>0.
$$
This shows (use also basic identities of the Gamma function given in Formula (1.2.1) and (1.2.2) in \cite{L}) that for $u$ close to $0$,
$$\E\left\{e^{iuU_T}\right\}=1-cu^{-2\nu} + o(u^{-2\nu}),$$
for some constant $c>0$, which proves our claim.
\end{proof}


\section{A weak law of large numbers} The next result answers the
second part of Benjamini's original question:
\begin{theo}
Let $(S_n,n\ge 0)$ be the process on $\Z$ starting from $0$, which at a first
visit to a site makes a discrete symmetric Cauchy jump and at other visits
makes a $\pm 1$ Bernoulli jump. Then
\beq
\frac 1n \sup_{t \le n} \mid S_t \mid  \to 0 \text{ in probability}.
\label{W12}
\eeq
\end{theo}
\begin{proof}
We shall first prove \eqref{W12} with $S_t$ replaced by the auxiliary process
$\wt S_t$ which makes a discrete symmetric Cauchy jump at
a first visit to a site and at other visits makes a jump with distribution $\wt P_1$, where
\begin{eqnarray*}
\wt P_1\{\pm1\} = 1/4,\ \wt P_1\{0\}=1/2 \text{ and }\wt P_1\{u\} = 0
\text{ for } u \notin \{-1,0,+1\}.
\end{eqnarray*}
Quantities referring to the walk $\{\wt S_n\}$ will all be
decorated with a tilde, but will otherwise be defined in the same
way as their analogues without a tilde.
We further remind the reader that $P_1$ is
the distribution on $\Z$ which puts mass $1/2$ on $\pm 1$ and that
$P_2$ is the discrete Cauchy distribution.

Let $\wt R_n$ be the range at time $n$, i.e.,
\begin{equation} \label{wtR}
\wt R_n = \text{ cardinality of } \{\wt S_0,\wt S_1,\dots,\wt S_{n-1}\}.
\end{equation}
This means that during the time interval $[0,n]$, $ \wt S_\l$ took
exactly $\wt R_n$ Cauchy jumps and $n-\wt R_n$ steps with distribution $\wt P_1$.
Let us now use the construction of the $\{\wt S_\l\}$ which is the
analogue of the one given for $\{S_\l\}$ in the introduction. More precisely, let $(\wt Y(1,\l))_{\l\ge 1}$ and $(\wt Y(2,\l))_{\l\ge 1}$ be independent sequences of independent random variables respectively of law $\wt P_1$ and $P_2$, then $(\wt S_n)_{n\ge 1}$ is such that $\wt S_0=0$ and for $n\ge 1$,
\begin{eqnarray*}
\wt S_n = \sum_{\l=1}^{n-\wt R_n} \wt Y(1,\l) + \sum_{\l =1}^{\wt R_n}
Y(2,\l).
\end{eqnarray*}
with $\wt R_n$ defined by \eqref{wtR}.

Consequently, for any $\ep > 0$ it holds that on the
event $\{\wt R_n \le \ep n\}$,
\begin{eqnarray*}
\sup_{t \le n} |\wt S_t| \le \sup_{s \le n} \Big|\sum_{\l =1}^s \wt Y(1,\l)\Big| + \sup_{r \le\ep n}\Big|\sum_{\l
=1}^r  Y(2,\l)\Big|.
\end{eqnarray*}
By maximal inequalities
(see \cite{Bi}, Theorem 22.5) we therefore have for any $\ep \le
1, \al
> 0$,
\begin{eqnarray}
\label{W3}
\nonumber \pp\{\sup_{t \le n} |\wt S_t| \ge 8 \al n\} &\le & \pp\{\wt R_n >\ep n\}
+  4\max_{t\le n}\pp\left\{\Big|\sum_{\l =1}^t \wt Y(1,\l)\Big| \ge \al n\right\} \\
   &+ & 4\max_{t\le \ep n} \pp\left\{\Big|\sum_{\l =1}^t Y(2,\l)\Big|\ge \al n\right\}.
\end{eqnarray}
Now, as is well known (eg. by Chebyshev's inequality), for each fixed $\al> 0$, \beq \label{W4} \max_{t \le
n}\pp\left\{\Big|\sum_{\l =1}^t \wt Y(1,\l)\Big| \ge  \al n\right\} \to 0 \text{ as } n \to \infty. \eeq Also, for
fixed $\al
> 0, \ep > 0, t \le \ep n$,
\beq \pp\left\{\Big|\sum_{\l=1}^t Y(2,\l)\Big| \ge \al n\right\} \le \pp\left\{ \Big|\sum_{\l=1}^t  Y(2,\l)\Big|
\ge \frac \al\ep t\right\}, \eeq and
 \beq \label{W5} \lim_{t \to \infty} \pp\left\{\Big|\sum_{\l
=1}^t   Y(2,\l)\Big| \ge \frac \al \ep t\right\} = f(\frac \al \ep), \eeq for some function $f(\cdot)$. Moreover,
$f(\al/ \ep)$ can be made as small as desired by taking $\al/\ep$ large. In fact,
$$
\frac 1m \sum_{\l=1}^m Y(2,\l) \text{ converges in distribution to a
Cauchy variable},
$$
as $m \to \infty$ (see \cite{F}, Theorem 17.7). It is immediate
from \eqref{W3}-\eqref{W5} that
\beq
\frac 1n \wt R_n \to 0 \text{ in
probability } \label{t2}
\eeq
is a sufficient condition for \eqref{W12} with $S_t$ replaced by $\wt S_t$.

We now turn to a proof of \eqref{t2}. Since $0 \le \wt R_n/n \le 1$,
\eqref{t2} is equivalent to
$$
\frac 1n \E[\wt R_n] = \frac 1n \sum_{t=0}^{n-1} \pp\left\{A_{t,n-t}\right\}
\to  0,
$$
where
\begin{eqnarray*}
A_{t,\l} &= \{\wt S_t \text{ is not revisited during
}[t+1,t+\l-1]\}\\
&=\{\wt S_t \ne \wt S_{t+s} \text{ for } 1 \le s \le \l-1\}.
\end{eqnarray*}
In particular, since $A_{t,\l}$ is decreasing in $\l$,
a sufficient condition for the WLLN for $\wt S_n$ is that
\beq
\lim_{\l \to \infty} \pp\left\{A_{t,\l}\right\} =0 \text{ uniformly in } t.
\label{t3}
\eeq
Recurrence essentially is property \eqref{t3}, {\it without the uniformity requirement}. To prove \eqref{t3} with
the uniformity we use the coupling defined in the remark below Example \ref{ex1}, as we now explain. Let $Q_0$ be
the transition probability matrix of a simple random walk on $\Z^2$. Denote this walk by $\{(U_n,V_n)\}_{n \ge 0}$
and let its starting point be $(0,0)$.
We proved in Example \ref{ex1} that $Q_0$ is successfully coupled with
itself. We shall use a part of that result here. We also need to
know that there exists another process $\{(U_n,W_n)\}_{n
\ge 0}$ which also starts at $(0,0)$ and takes values in $\Z^2$ and
in addition a coupled process $\{(U_n,V_n,W_n)\}_{n \ge 0}$ such that
\begin{eqnarray*}
\begin{array}{l}
\text{the law of the imbedded process of $\{(U_n,W_n)\}$ in
the}\\
\text{$U$-axis is the same as the law of Benjamini's process } \{\wt S_n\},
\end{array}
\end{eqnarray*}
and
\beq
\mid W_n \mid \le \mid V_n\mid +1.
\label{wc2}
\eeq
We remind the reader that the imbedded process here is $\{U_{\tau_n}\}_{n
\ge 0}$, where $\tau_0=0$ and for $\l \ge 1$
$$
\tau(\l) = \inf\{t > \tau(\l-1): W_t = 0\}.
$$
In the remainder of this proof we shall often write $\Ga(\phi)$
instead of $\Ga_\phi$ for certain $\Ga$ and $\phi$, in order to avoid double subscripts.
Now fix some $t \in \{0,1,\dots,n-1\}$. For time running from 0 to
$t$ we let $U^t_0,U^t_1,\dots,U^t_t$ be a copy of $\{\wt
S_\l\}_{0\le \l\le t}$. No coupling of this process with another process is
needed. However, we shall further need an independent copy of the
variables
$\{(U_n,V_n,W_n)\}_{n \ge 0}$ with its corresponding sequence of
times $\tau_\l$ at which the walk $\{(U_n,W_n)\}$ visits the $U$-axis. The
successive positions of Benjamini's walk determined by the triple
$\{(U_n,V_n,W_n)\}_{n \ge 0}$ itself would be $U(\tau_0),
U(\tau_1),\dots$. However we want to shift those positions to come
right after the previous points $\{U_\l^t\}$. This requires one
important change. In the coupling construction by itself,
at a time $\tau$ at which $W_\tau = 0$, assume that the Benjamini walk
arrived in some point,
$u$ say, on the $U$-axis. In order to choose  the next step for the walk
one must now decide whether the visit to $u$ at $\tau$ is the first
visit by the walk to $u$ or not. In the construction of
Example \ref{ex1} it would be  a first visit if and only if  $U(\tau_m) \ne u$ for
$0 \le \tau_m < \tau$. Here we have to modify this. We think of the
walk as first traversing $U^t_0,U^t_1,\dots,U^t_t$, and then
to start from time $t$ on to use the coupling construction. The
visit at time $\tau$ to $u$ will therefore be counted as the first
visit if and only if
$$
U(\tau_m) \ne u-U_t^t \text{ for $0 \le \tau_m < \tau$ and
$U_s^t \ne u$ for } 0 \le s \le \tau.
$$
After this change, the path
$$
U^t_0, U^t_1,\dots, U^t_t = U^t_t+ U(\tau_0), U^t_t+ U(\tau_1),
U^t_t +U(\tau_2),\dots
$$
is a typical path of a Benjamini walk, but with a modified rule
for determining whether a point is fresh or old. To be more precise, let
$\Th = \Th(t)= \{U^t_0 = 0, U^t_1,\dots,U^t_t\}$ be the set of
points visited by the Benjamini walk during $[0,t]$. Now first fix $\Th(t)$. Then
$A_{t,\l}$ occurs if and only if none
of the next $\l-1$ positions of a Benjamini walk equals $U^t_t$.
However, for this second stage the points of $\Th$ are regarded as old
points, even if they have not been visited before. Thus we use a
modified Benjamini walk in which the walk takes a simple symmetric walk
step when it is at an old point or a point from $\Th$, and a
discrete Cauchy distribution when the walk is at a fresh point
outside $\Th$. We shall call this the $\Th$-{\it modified walk}. The
original Benjamini walk is the special case of this when $\Th =
\emptyset$. When the dependence on $\Th$ is important we shall
indicate this by a superscript $\Th$. In particular, the law of
the walk which we just described (in which we regard the points of
$\Th$ as old points) is written as $\pp ^\Th$.
Choosing or modifying $\Th$ merely modifies the rule by which the
index $i$, or equivalently the function $F$ in \eqref{3.18} and \eqref{index} is
chosen. However, Lemma \ref{lem3.2} remains valid for the $\Th$-modified
process. In particular, we can express the conditional law of $\wt
S_{t+s} - \wt S_t$ given $\kF_t$, by means of $\pp ^{\Th(t)}$. This
gives
\begin{eqnarray*}
\pp \{A_{t,\l}\} &=& \E\big\{\pp\{\wt S_{t+q}\neq\wt S_t \text{ for all } 1 \le q \le \l-1 \big| \kF_t\}\big\}\\
&=& \E \big\{\pp ^{\Th(t)}\{ \wt S_q \ne 0
\text{ for all } 1 \le q \le \l-1\}\big\}\\
& \le & \sup_\Th \pp ^\Th
\{\wt S_q \ne 0 \text{ for all } 1 \le q \le \l-1\}.
\end{eqnarray*}

We now complete the proof of \eqref{t3}.
We find it useful for this purpose to introduce the events
\begin{eqnarray*}
\kA_q := \{U_q=V_q=0\} \cap \{V_{q+1} = -1\}.
\end{eqnarray*}
Since $\{(U_q,V_q)\}_{q \ge 0}$ is a simple random walk it is well known that
this walk is recurrent, so that with
probability 1, the event $\{U_q=V_q=0\}$ occurs for infinitely many $q$.
By a straightforward application of a
conditional version of the Borel-Cantelli lemma (cf. Theorem 12.15 in \cite {W})
it then follows that, again with
probability 1, $\kA_q$ occurs infinitely often. For every $q$ with $V_q = 0$
we have $W_q \in \{-1,0,1\}$, by
virtue of \eqref{wc2}. The remark following Example \ref{ex1} now shows
that if $\kA_q$ occurs for some $q$, then also
\begin{eqnarray*}
\kB_q := \left\{ U(q)=W(q)=0 \right\}\cup \left\{ U(q+1)=W(q+1)=0 \right\}
\end{eqnarray*}
occurs for the same $q$.
Note that the event 
$$\{\kB_q \text{ fails for all } 1\le q \le \l-1\}$$
coincides with the event
$$\cap_{q=1}^\l \{U(q)\neq 0 \hbox{ or }W(q)\neq 0\}.$$
Note now that (since $\tau_\l\ge \l$ and since $W(q)=0$ implies $q=\tau_r$ for some $r$) the event $\cap_{q=1}^\l \{U(q)\neq 0 \hbox{ or }W(q)\neq 0\}$ occurs when the event $\{\wt{S}_t=U(\tau_t)\ne 0 \hbox{ for all } 1\le t\le \l\}$.
Thus we have
\begin{eqnarray*}
\pp ^\Th\{\wt S_q \ne 0 \text{ for } 1 \le q \le \l\}
&\le & \pp ^\Th \{\kB_q \text{ fails for all }1\le q \le \l-1\}\\
&\le & \pp \{\kA_q \text{ fails for all }1\le q \le\l-1\}
\end{eqnarray*}
(use contrapositives for the last inequality).
But the right hand side here is independent of $t$ and $\Th$,
since it involves only the simple random walk $(U_n,V_n)$.
In addition this right hand side tends to 0 as $\l \to \infty$, since we already proved
that with probability 1 infinitely many $\kA_q$ occur. This last
estimate is uniform in $t,\Th$, as desired.

\comment{
has been reduced to showing that $\kA_q$ implies
$\kB_q$, which we give now.

If $V_q = 0$, then also $\mid W_q  \mid \le 1$ , by virtue of \eqref{wc2}.
Thus, if $\kA_q$ occurs for some $q$, then also $U_q =
0, W_q \in \{-1, 0, +1\}$ occurs.
If $W_q= 0$, then $q$ also equals $\tau_r$ for some $r$
(by the definition of the $\tau$'s), so that
$U_q = U(\rho_r) = W_q = W(\tau_r)=0$ occurs.
This leaves us with the case that $\kA_q \cap \{W_q = \pm 1\}$ occurs.
Assume next that $\kA_q \cap \{W_q =1\}$ occurs.
By the definition
of $\kA_q$ and of the successful coupling of $Q_0$ with itself, we
then also have $W_{q+1} - W_q = V_{q+1} -V_q
= -1$, and thus $W_{q+1}= 1 -1 = 0$. Finally, $U_{q+1}=U_q=0$, since
the coupling in Example \ref{ex1}
 restricts the change in $(U,V)$ to a horizontal or vertical step of
one unit.
As a few lines back this gives us that $q+1$ equals some $\rho_r$
and $U(\tau_r) = W(\tau_r) =0$ as well as $\kB_q$ occur.
Finally assume that $\kA_q \cap \{W_q = -1\}$ occurs.
Again we see that then $U_q=U_{q+1} = 0$, $V_{q+1}-V_q
=-1$, $W_{q+1} = W_q - (V_{q+1}-V_q) = 0$ (because $W_q=-1$).
Consequently $W_{q+1}=0 =W(\tau_r)$ for some $\tau_r$.
We have therefore proved that in all cases $\kA_q \subset
\kB_q$, as desired.
}

This finally proves \eqref{t3} and the WLLN, i.e., \eqref{W12} with $S_t$
replaced by $\wt S_t$. However, this proof is
for the $\{\wt X_n\}$-process which takes a
step with distribution $\wt P_1$ whenever the walk is at an old
point. We shall now show that this implies the WLLN for
the process $\{S_n\}$, i.e., \eqref{W12} itself.
Indeed, in the notation of the proof of Lemma \ref{lemP0}, the processes
$\{S_n,i(n)\}_{n \ge 1}$ and $\{\wt S_{\rho_n},\wt i(\rho_n)\}_{n \ge 1}$
have the same law. In particular,
\begin{eqnarray}
\left(\frac 1t \sup_{\l \le t} \mid S_\l \mid, i(n)\right) \text{ and }
\left(\frac 1t \sup_{\l \le t} \mid \wt S_{\rho_\l} \mid, \wt
i(\rho_n)\right)
\label{wc15}
\end{eqnarray}
have the same law. As explained in the proof of Lemma \ref{lemP0} we may
even assume that all these variables are defined on the same
probability space of sequences $\{A_\l\}_{\l \ge 0},\{B_\l\}_{\l
\ge 0}$ provided with the measure which makes all these
variables i.i.d. uniform on $[0,1]$. We denote this probability
measure by $\pp$. It follows from from the definition of the sequence $\rho_\l$ that for $q\ge 1$
$$
\pp\{\rho_{\l+1} -\rho_\l \ge q \mid \si (\rho_1,\rho_2, \dots, \rho_\l)\}
\le 2^{-(q-1)}.
$$
In turn, this implies that for some
constant $C \in (0,\infty)$,
\begin{eqnarray*}
\limsup_{t \to \infty} \sup_{\l \le t} \frac {\rho_\l}t =
\limsup_{t \to \infty} \frac {\rho_t}t
=\limsup_{t \to \infty} \frac {\sum_{\l=1}^t [\rho_\l - \rho_{\l-1}]}t
\le C \text{ with probability 1.}
\end{eqnarray*}
It follows that
\begin{eqnarray*}
\begin{array}{l}
\pp\{\sup_{\l \le t} \mid \wt S_{\rho_\l} \mid >\ep t\}\\
\le \pp\{\max_{\l \le t} \rho_\l > (C+1)t\} +
\pp\{\sup_{\l \le (C+1)t}\mid \wt S_\l \mid > \ep t\}.
\end{array}
\end{eqnarray*}
The last term on the right here tends to 0 as $t \to \infty$ and
$\ep > 0$ fixed, by virtue of \eqref{W12} with $S_t$ replaced by
 $\wt S_t$. Since also the first
term on the right here tends to 0 as $t \to \infty$, we conclude from
\eqref{wc15} that $(1/t) \sup_{\l \le t}\mid S_\l \mid \to 0$ in probability,
as desired.
\end{proof}


\end{document}